\documentclass{amsart}
\usepackage{amsfonts,amssymb,amscd,amsmath,enumerate,verbatim,calc}
\usepackage[all]{xy}

\newcommand{\CM}{Cohen-Macaulay}

\newcommand{\n}{\mathfrak{n} }
\newcommand{\m}{\mathfrak{m} }

\newcommand{\Ef}{\mathfrak{E} }

\newcommand{\C}{\mathcal{C} }
\newcommand{\D}{\mathcal{D} }
\newcommand{\rt}{\rightarrow}

\newcommand{\ov}{\overline}

\newcommand{\wh}{\widehat }

\newcommand{\image}{\operatorname{image}}

\newcommand{\height}{\operatorname{height}}

\newcommand{\mds}{\operatorname{\underline{mod}}}

\newcommand{\CMS}{\operatorname{\underline{CM}}}

\theoremstyle{plain}

\newtheorem{theorem}{Theorem}[section]
\newtheorem{corollary}[theorem]{Corollary}
\newtheorem{lemma}[theorem]{Lemma}

\theoremstyle{definition}

\newtheorem{s}[theorem]{}
\newtheorem{remark}[theorem]{Remark}

\theoremstyle{remark}

\begin{document}

\title[Stable equivalences]{Equivalences of stable categories of Gorenstein local rings}
\author{Tony~J.~Puthenpurakal}
\date{\today}
\address{Department of Mathematics, IIT Bombay, Powai, Mumbai 400 076}

\email{tputhen@math.iitb.ac.in}

 \begin{abstract}
In this paper we give bountiful examples of Gorenstein local rings $A$ and $B$ such that there is a triangle equivalence between the stable categories \underline{CM}($A$), \underline{CM}($B$).
\end{abstract}
 \maketitle
\section{introduction}
Let $R$ be a commutative Artin ring and let $\Lambda$ be a not-necessarily commutative Artin $R$-algebra. Let $\mds(\Lambda)$ denote the stable category of $\Gamma$.
The study of equivalences of stable categories of Artin $R$-algebras has a rich history; see \cite[Chapter X]{ARS}. Auslander discovered that many concepts in representation theory of Artin algebras have
natural analogues in the study of maximal \CM \ ( = MCM)  modules of a commutative \CM \  local ring $A$. See \cite{Y} for a nice exposition of these ideas.
By following Auslander's idea  we investigate equivalences of the stable
category of MCM modules over commutative \CM \ local rings. If $A$ is Gorenstein local then, $\CMS(A)$,  the stable category of MCM $A$-modules has a triangulated structure, see \cite[4.4.1]{Buchw}. So as a first
iteration in this program,  we may investigate triangle equivalences of stable categories of MCM modules over commutative Gorenstein  local rings.

There is paucity of examples of Gorenstein local rings $A, B$ such that $\CMS(A)$ is triangle equivalent to $\CMS(B)$. There are two well-known examples of triangle equivalences. First is
Kn\"{o}rrer periodicity \cite{Kno} for hypersurfaces. Another gives an equivalence $\CMS(A) \rt \CMS(\widehat{A})$ where $A$ is a excellent, henselian, Gorenstien isolated singularity, see \cite[A.6]{KMB}. In this paper we give bountiful examples of Gorenstein local rings $A$ and $B$ such that there is a triangle equivalence between the stable categories $\CMS(A), \CMS(B)$.

\emph{The Example:} \\
Let $(A,\m)$  be a Gorenstein local ring of dimension $d \geq 1$. Assume $A$ is essentially of finite type over a field $K$ and that $A$ is an isolated singularity. Also assume that the completion of $A$ is a domain.  Further assume that the Grothendieck group of the completion of $A$ is a finitely generated abelian group. \\
Let $\Ef_P(A)$ be the set of isomorphism classes of pointed etale neighborhoods of $A$, see \ref{e}, \ref{ep}.
It is well-known that  $\Ef_P(A)$ is a directed set and  $A^h$, the henselization of $A$,
is
\[
A^h = \lim_{R \in  \Ef_P(A)}R.
\]
By Remark \ref{inf-chain}, for any $B \in \Ef_P(A)$  there exists atleast one infinite chain in   $\Ef_P(A)$ starting at $B$.

1. Let $A_1 \rt A_2 \rt \cdots \rt A_n \rt \cdots$ be a chain in  $\Ef_P(A)$. Then we show that there exists $n_0$ such that
$$ \CMS(A_n) \cong \CMS(A_{n+1})  \quad \text{for all $n \geq n_0$}. $$

2. There exists a chain $B_1 \rt B_2 \rt \cdots \rt B_n \rt \cdots$  in  $\Ef_P(A)$
such that
$$ \CMS(B_n) \cong \CMS(A^h)  \quad \text{for all $n \geq 1$}. $$

\begin{remark}
In (2) note that $B_n$ is NOT henselian, see  \ref {h-N-ft}. Nevertheless  \\ $\CMS(B_n) \cong \CMS(A^h)$ is Krull-Schmidt  for all $n \geq 1$.
\end{remark}
\emph{Technique used to prove the result:}\\
Under the hypotheses of the example it can be shown that $G(\CMS(A^h))$ is a finitely generated abelian group, see \ref{hensel}.
Let $R, S \in \Ef_P(A)$ and let $R \xrightarrow{\psi} S$ be a homomorphism  in  $ \Ef_P(A)$. We show there is a commutative diagram
\[
\xymatrix{
\
&G(\CMS(S))
\ar@{->}[dr]^{f_S}
 \\
G(\CMS(R))
\ar@{->}[rr]_{f_R}
\ar@{->}[ur]^{G(\psi)}
&\
&G(\CMS(A^h))
}
\]
Here $G(\psi), f_R, f_S$ are injective  by \ref{etale-extn} and \ref{etale-extn-2}. We show that if  $G(\psi)$ is an isomorphism then $\psi \colon \CMS(R) \rt \CMS(S)$ is an equivalence, see \ref{wc-equi}.

(1) We note that if  $A_1 \rt A_2 \rt \cdots \rt A_n \rt \cdots$ be a chain in  $\Ef_P(A)$
then we have an ascending chain of subgroups of $G(\CMS(A^h))$
$$\image f_{A_1} \subseteq \image f_{A_2}  \subseteq \cdots \subseteq \image f_{A_n}  \subseteq \cdots$$
As $G(\CMS(A^h))$ is a finitely generated abelian group it follows that there exists $m$ such that $\image f_{A_n} = \image f_{A_m}$ for all $n \geq m$. It follows that
$$ \CMS(A_n) \cong \CMS(A_{n+1})  \quad \text{for all $n \geq m$}. $$

(2) As $G(\CMS(A^h))$ is a finitely generated abelian group we first show that there exists $B \in \Ef_P(A)$ such that the natural map $f_B \colon \CMS(B) \rt  \CMS(A^h)$ is surjective.
As it is injective, see \ref{etale-extn-2} it follows that $f_B$ is an isomorphism. So again by \ref{etale-extn-2} we get $\CMS(B)  \cong \CMS(A^h)$.

We choose an infinite chain
 $B_1 \rt B_2 \rt \cdots \rt B_n \rt \cdots$  in  $\Ef_P(A)$ starting at $B = B_1$. It follows from the commutative diagram above that $f_{B_n} \colon G(\CMS(B_n)) \rt G(\CMS(A^h))$ is an isomorphism for all $n \geq 1$. So  by \ref{etale-extn-2} we get $\CMS(B_n)  \cong \CMS(A^h)$ for all $n \geq 1$.

 \begin{remark}
 Let $(A,\m)$ be a Gorenstein local ring, essentially of finite type over a field $K$ and   that $A$ is an isolated singularity. Then by \cite[1.5]{KK} the natural map $G(A) \rt G(A^h)$ is injective. Assume $\wh{A}$ is also a domain (automatic if $\dim A \geq 2$).
 The entire argument as above goes through if we just assume the quotient group $G(\wh{A})/G(A)$ is a finitely generated abelian group. Regrettably we do  not have a nice class of rings with this property (i.e., with $G(\wh{A})$ infinitely generated abelian group and $G(\wh{A})/G(A)$ finitely generated abelian group).
 \end{remark}

Here is an overview of the contents of this paper. In section two we prove some general results on Grothendieck groups of triangulated categories. In section three we discuss some preliminary facts on pointed \'{e}tale neighbourhoods of a local ring. In section four we discuss our example. In the next section we give many examples of local rings with $G(\wh{A})$ finitely generated.
In the appendix we discuss a fact which is crucial for us.
 \section{Some generalities on Grothendieck groups of triangulated categories}
Throughout all triangulated categories considered will be skeletally small.
Let $\C, \D$ be triangulated categories.
\s
An triangulated  functor $F\colon \C\rt \D$ is called an \emph{equivalence up to direct summand's}
if it is fully faithful and any object $X \in \D$  is isomorphic to a direct summand of $F(Y)$ for
some $Y \in \C$.

\s \label{wc} We say $\C$ has \textit{weak cancellation}, if for $U, V, W \in \C$ we have
$$ U \oplus V \cong  U \oplus W \implies V \cong W. $$

\s Let $\C$ be a triangulated category. The Grothendieck group $G(\C)$ is the quotient group of the free abelian group
on the set of isomorphism classes of objects of $\C$  by the Euler relations: $[V] =
[U] + [W]$ whenever there is an exact triangle in $\C$
$$U \rt V \rt W \rt U[1]. $$
As $[U[1]] = -[U]$ in $G(\C)$, it follows that any element of $G(\C)$ is of the form $[V]$ for some $V \in \C$.

We first show:
\begin{theorem}
\label{main} Let $\phi \colon \C \rt \D$ be a triangulated functor which is an  equivalence upto direct summands. Then the natural map $G(\phi) \colon G(\C) \rt G(\D)$ is injective.
\end{theorem}
\begin{proof}
Suppose $G(\phi)([U]) = [\phi(U)] = 0$ in $G(\D)$. So by \cite[2.4]{Th} there exists $X, Y, Z \in \D$ and triangles
\begin{align*}
X &\rt Y\oplus \phi(U) \rt Z \rt X[1], \quad \text{and} \\
X &\rt Y \rt Z \rt X[1].
\end{align*}
There exists $X_1, Z_1 \in \D$ such that $X \oplus X_1 = \phi(M)$ and $Z \oplus Z_1 = \phi(N)$. We add the triangles $X_1 \xrightarrow{1} X_1 \rt 0 \rt X_1 \rt 0 \rt X_1[1]$ and $0 \rt Z_1 \xrightarrow{1} Z_1 \rt 0$ to the above triangles.

The modified second triangle is
\[
\phi(M) \rt Y^\prime \rt \phi(N) \rt \phi(M)[1].
\]
It follows that $Y^\prime  = \phi(L)$ for some $L \in \C$.
The modified first triangle becomes
$$ \phi(M) \rt\phi(L)\oplus \phi(U) \rt \phi(N) \rt \phi(M)[1].$$
Note $\phi(L) \oplus \phi(U) = \phi(L\oplus U)$. We now use the fact that if $W_1, W_2 \in \C$ then $W_1 \cong W_2$ if and only if $\phi(W_1) \cong \phi(W_2)$. As $\phi$ is fully faithful we get triangles in $\C$
\begin{align*}
M &\rt L \oplus U \rt N \rt M[1], \quad \text{and} \\
M &\rt L  \rt N \rt M[1].
\end{align*}
It follows that $[U] = 0$ in $G(\C)$.
\end{proof}

A natural question is when is $Y \in \D$ is isomorphic to $\phi(X)$ for some $X \in \C$. A necessary condition is of course that $[Y] \in \image G(\phi)$. Some what surprisingly the converse also holds if $\D$ satisfies weak cancellation.
\begin{theorem}
\label{main-onto}
(with hypotheses as in Theorem \ref{main}). Also assume $\D$ satisfies weak cancellation.
If $[V] \in \image G(\phi)$ then there exists $W \in \C$ with $\phi(W) \cong V$.
\end{theorem}
We need the following:
\begin{lemma}\label{phi-sum}
(with hypotheses as in Theorem \ref{main}). Let $U, V \in \C$. If $\phi(U)$ is a direct summand of $\phi(V)$ then $U$ is a direct summand of $V$.
\end{lemma}
\begin{proof}
We have a diagram
 \[
\xymatrix{
\
&\phi(U)
\ar@{->}[dr]^{h}
 \\
\phi(V)
\ar@{->}[rr]_{f}
\ar@{->}[ur]^{g}
&\
&\phi(V)
}
\]
with $f^2 = f$ and $g\circ h = 1_{\phi(U)}$. As $\phi$ is fully faithful  there exists $f_1 \colon V \rt V$, $g_1 \colon V \rt U$ and $h_1 \colon U \rt V$ with $\phi(f_1) = f$, $\phi(g_1) = g$ and $\phi(h_1) = h$. Again as $\phi$ is fully faithful we get
$h_1\circ g_1 = f_1$, $f_1^2 = f_1$ and  $g_1\circ h_1 = 1_{U}$. The result follows.
\end{proof}
We now give
\begin{proof}[Proof of Theorem \ref{main-onto}]
Say $[V] = [\phi(U)]$ in $G(\D)$. By \cite[2.4]{Th} we have triangles in $\D$
\begin{align*}
X &\rt Y\oplus \phi(U) \rt Z \rt X[1], \quad \text{and} \\
X &\rt Y\oplus V \rt Z \rt X[1].
\end{align*}
There exists $X_1, Z_1 \in \D$ such that $X \oplus X_1 = \phi(M)$ and $Z \oplus Z_1 = \phi(N)$. We add the triangles $X_1 \xrightarrow{1} X_1 \rt 0 \rt X_1 \rt 0 \rt X_1[1]$ and $0 \rt Z_1 \xrightarrow{1} Z_1 \rt 0$ to the above triangles.
Set $Y^\prime  = Y \oplus X_1 \oplus Z_1$.

The (modified) first triangle yields
$Y^\prime \oplus \phi(U) \cong \phi(E)$ for some $E \in \C$. By \ref{phi-sum} we have that $U$ is a direct summand of $E$. Say $E = U \oplus U_1$. So $Y^\prime\oplus \phi(U) \cong \phi(U) \oplus \phi(U_1)$. As $\D$ satisfies weak cancellation we get $Y^\prime\cong \phi(U_1)$.

 The (modified) second  triangle yields
$\phi(U_1) \oplus V \cong \phi(L)$ for some $L \in \C$. By \ref{phi-sum} we have that $U_1$ is a direct summand of $L$. Say $L = U_1 \oplus U_2$. So we obtain
$$\phi(U_1) \oplus V = \phi(U_1) \oplus \phi(U_2).$$
As $\D$ satisfies weak cancellation we get  $V \cong  \phi(U_2).$ The result follows.
\end{proof}
The following consequence of Theorem \ref{main-onto} is significant.
\begin{corollary}\label{equi}
(with hypotheses as in Theorem \ref{main-onto}). If $G(\phi)$ is an isomorphism then $\phi$ is an equivalence.
\end{corollary}
\begin{proof}
By our assumption $\phi$ is fully faithful. By Theorem \ref{main-onto} it follows that $\phi$ is dense. So $\phi$ is an equivalence.
\end{proof}
Another surprising consequence of Theorem \ref{main-onto} is the following:
\begin{corollary}\label{image}
(with hypotheses as in Theorem \ref{main-onto}).  Let $Y \in \D$. Then there exists $X \in \C$ with $\phi(X) = Y \oplus Y[1]$.
\end{corollary}
\begin{proof}
Notice in $G(\D)$ we have  $$[Y \oplus Y[1]] = [Y] + [Y[1]] = [Y] - [Y] = 0.$$
The result follows from Theorem \ref{main-onto}.
\end{proof}

 \section{pointed  \'{e}tale neighborhoods}
In this section we recall  definition of pointed \'{e}tale extensions and discuss a few of its properties.

 \s \label{e} A local homomorphism $ \psi \colon (A,\m) \rt (B,\n) $ of
local rings is unramified provided $B$ is essentially of finite type over $A$
(that is, $B$ is a localization of some finitely generated $A$-algebra) and
the following properties hold.\\
(i) $\m B  = \n$, and\\
(ii) $B/\m B$ is a finite separable field extension of $A/\m$.\\
If, in addition, $\psi$ is flat, then we say $\psi$ is \'{e}tale. (We say also that $B$
is an unramified, respectively, \'{e}tale extension of $A$.) Finally, a pointed
\'{e}tale neighborhood  of $A$ is an \'{e}tale extension $(A,\m) \rt (B, \n)$ inducing
an isomorphism on residue fields.

\s\label{ep} The isomorphism classes
of pointed  \'{e}tale neighborhoods of a local ring $(A,\m)$ form a direct
system (see \cite[Chapter 3]{I} for details). This implies that  if $A \rt B$
and $A \rt C$ are pointed   \'{e}tale  neighborhoods then there is at most
one homomorphism  $B \rt C$ making the obvious diagram commute.
Let $\Ef_P(A)$ be the set of isomorphism classes of pointed \'{e}tale neighborhoods of $A$.
Then $A^h$, the henselization of $A$,
is
\[
A^h = \lim_{R \in  \Ef_P(A)}R.
\]

\begin{remark}\label{inf-chain}
Assume $A$ is \CM  \ and essentially of finite type over a field $K$. Also assume $\wh{A}$ is a domain.
By \ref{h-N-ft} it follows that $A^h$ is NOT of essentially finite type over $K$. Thus $\Ef_P(A)$ contains infinitely many terms. So there exists atleast one infinite chain in $\Ef_P(A)$. Note we may assume this infinite chain starts at $A$

Let $B \in \Ef_P(A)$. Then $B^h = A^h$ and we may assume $\Ef_P(B) \subseteq \Ef_P(A)$. By argument as before there exists an infinite  chain in $\Ef_P(B)$ (and so in $\Ef_P(A)$) starting at $B$.
\end{remark}

\s\label{etale-extn} Let $A$ be an excellent Gorenstein isolated singularity. Let $B, C \in \Ef_P(A)$ and let $\psi \colon B \rt C$ be a morphism in $\Ef_P(A)$.  Note $B, C$ are also an excellent Gorenstein isolated singularity. Note $\psi$  is flat, see \cite[Chapter 1, 2.7]{I}, and so induces a   functor
$\psi \colon \CMS(B) \rt \CMS(C)$ which is clearly fully faithful. Furthermore as $C$ is  \'{e}tale over $B$ we get that  every MCM $C$-module $M$ is a direct summand of $N \otimes_B C$ where $N$ is a MCM
$B$-module, see \cite[10.5, 10.7]{LW}. By Theorem \ref{main} the natural map
$G(\psi) \colon G(\CMS(B)) \rt  G(\CMS(C))$ is injective.

\begin{lemma}\label{weak-cancellation}
Let $(A,\m)$ be a Gorenstein local ring. Then $\CMS(A)$ has weak cancellation (see \ref{wc})
\end{lemma}
\begin{proof}
Suppose $M, N, L$ are MCM $A$-modules such that $M\oplus N \cong M \oplus L$ in $\CMS(A)$.
Then as $A$-modules $M \oplus N \oplus A^r \cong M \oplus L \oplus A^s$ for some $r, s$. By \cite[1.16]{LW} it follows that $N \oplus A^r \cong L \oplus A^s$ as $A$-modules. So $N \cong L$ in $\CMS(A)$.
\end{proof}

\s\label{wc-equi}(with hypotheses as in \ref{etale-extn}) If $G(\psi) \colon G(\CMS(B)) \rt G(\CMS(C))$ is an isomorphism then by \ref{equi} and
\ref{weak-cancellation},
 we get that $\psi \colon \CMS(B) \rt \CMS(C)$ is an equivalence.

\s \label{hensel} Let $(A,\m)$ be an excellent Gorenstein isolated singularity. Then $A^h$ is also an excellent Gorenstein isolated singularity, see \cite[10.7]{LW}. In particular we have an equivalence of categories $\CMS(A^h) \cong \CMS(\wh{A})$, by \cite[A.6]{KMB}. If $G(\wh{A})$ is finitely generated abelian group then $G(\CMS(\wh{A}))  = G(A)/([A])$ is finitely generated abelian group. It follows that $G(\CMS(A^h))$ is finitely generated abelian group.

\s\label{etale-extn-2} Let $A$ be an excellent Gorenstein isolated singularity. Let $B \in \Ef_P(A) $. Then the map $f \colon B \rt A^h$ is flat. It is readily verified that $f \colon \CMS(B) \rt \CMS(A^h)$ is fully faithful (use $B^h = A^h$). Also as $f \colon B \rt A^h$ is a direct limit of \'{e}tale-extensions every MCM $A^h$-module $M$ is a direct summand of $N \otimes_B A^h$ where $N$ is a MCM
$B$-module, see \cite[10.5, 10.7]{LW}. By Theorem \ref{main} the natural map
$f_B \colon G(\CMS(B)) \rt  G(\CMS(A^h))$ is injective.  If $f_B$ is an isomorphism then by \ref{equi} and
\ref{weak-cancellation},
 we get that $f \colon \CMS(B) \rt \CMS(A^h)$ is an equivalence.

\section{The example}
Let $(A,\m)$  be a Gorenstein local ring of dimension $d \geq 1$. Assume $A$ is essentially of finite type over a field $K$ and  that $A$ is an isolated singularity. Also assume that the completion of $A$ is a domain. Further assume that the Grothendieck group of the completion of $A$ is a finitely generated abelian group. \\
Let $\Ef_P(A)$ be the set of isomorphism classes of pointed etale neighborhoods of $A$.
It is well-known that  $\Ef_P(A)$ is a directed set and  $A^h$, the henselization of $A$,
is
\[
A^h = \lim_{R \in  \Ef_P(A)}R.
\]
By \ref{hensel} it follows that $G(\CMS(A^h))$ is a finitely generated abelian group.

1. Let $A_1 \rt A_2 \rt \cdots \rt A_n \rt \cdots$ be a chain in  $\Ef_P(A)$. \\
\emph{Claim:}\\
 There exists $n_0$ such that
$$ \CMS(A_n) \cong \CMS(A_{n+1})  \quad \text{for all $n \geq n_0$}. $$

\emph{Proof of Claim:}\\
Let $R, S \in \Ef_P(A)$ and let $R \xrightarrow{\psi} S$ be a homomorphism  in  $ \Ef_P(A)$.   Note as $\psi, f_R \colon R \rt A^h$ and $f_S \colon S \rt A^h $  are flat we get a commutative diagram
\[
\xymatrix{
\
&G(\CMS(S))
\ar@{->}[dr]^{f_S}
 \\
G(\CMS(R))
\ar@{->}[rr]_{f_R}
\ar@{->}[ur]^{G(\psi)}
&\
&G(\CMS(A^h))
}
\]
Here $G(\psi), f_R, f_S$ are injective by  \ref{etale-extn} and \ref{etale-extn-2}.  By \ref{wc-equi}  it follows that if  $G(\psi)$ is an isomorphism then  $\psi \colon \CMS(R) \rt \CMS(S)$ is an equivalence.
We note that if  $A_1 \rt A_2 \rt \cdots \rt A_n \rt \cdots$ is a chain in  $\Ef_P(A)$
then we have an ascending chain of subgroups of $G(\CMS(A^h))$
$$\image f_{A_1} \subseteq \image f_{A_2}  \subseteq \cdots \subseteq \image f_{A_n}  \subseteq \cdots$$
As $G(\CMS(A^h)) $ is a finitely generated abelian group it follows that there exists $m$ such that $\image f_{A_n} =  \image f_{A_m}$ for all $n \geq m$.
Thus the map \\ $G(\CMS(A_n)) \rt G(\CMS(A_{n+1}))$ is an isomorphism for all $n \geq m$. It follows that
$$ \CMS(A_n) \cong \CMS(A_{n+1})  \quad \text{for all $n \geq m$}. $$

2. We have $G(\CMS(A^h))$ is a finitely generated abelian group,  we may assume $[N_1], \ldots, [N_s]$ generate it as an abelian group. By proof of \cite[Theorem 10.7]{LW} we may assume there exists $A_i \in  \Ef_P(A)$ and  finitely generated $A_i$ module $N_i$ such that $N_i = A^h\otimes_{A_i} M_i$. As $\theta_i \colon A_i \rt A_h$ is flat (with zero dimensional fiber)  it follows from  \cite[23.3]{M}  that $N_i$ is a MCM $A_i$-module. As
$\Ef_P(A)$ is directed there exists $B \in \Ef_P(A)$ such that there exists maps $\gamma_i \colon A_i \rt B$ in  $\Ef_P(A)$ for all $i$.  Set $E_i = B \otimes_{A_i} N_i$. Then as $\gamma_i$ is flat (with zero dimensional fiber),  $E_i$ is MCM $B$-module. Furthermore $E_i \otimes_B A^h \cong  N_i$. It follows that $f_B \colon G(\CMS(B)) \rt G(\CMS(A^h))$ is surjective. As it is injective, see \ref{etale-extn-2} it follows that $f_B$ is an isomorphism. So again by \ref{etale-extn-2} we get $\CMS(B)  \cong \CMS(A^h)$.

Now let $B = B_1 \rt B_2 \rt \cdots \rt B_n \rt \cdots$ be an infinite chain in $\Ef_P(A)$, see \ref{inf-chain}. By the commutative diagram above we get $f_{B_n} \colon G(\CMS(B_n)) \rt G(\CMS(A^h))$ is an isomorphism for all $n \geq 1$. So  by \ref{etale-extn-2} we get $\CMS(B_n)  \cong \CMS(A^h)$ for all $n \geq 1$.

\section{rings with $G(\widehat{A})$ finitely generated abelian group}
In this section we give many examples of Gorenstein rings with $G(\widehat{A})$ finitely generated abelain group (and $\wh{A}$ a domain).

\s Let $R = k[t^{a_1}, \cdots, t^{a_m}]$ be a symmetric numerical semi-group ring. Then $R$ is Gorenstein, see \cite[4.4.8]{BH}. Let
$$ A = R_{(t^{a_1}, \cdots, t^{a_m})}.$$ The completion of $A$ is
$k[[t^{a_1}, \cdots, t^{a_m}]]$  which is a domain. Furthermore as $\wh{A}$ is one-dimensional we get that $G(\widehat{A})$ is a finitely generated abelian group.

\s Let $R = k[X_1, \ldots, X_n]^G$ where $G$ is a finite subgroup of $SL_n(k)$ with $|G|$ non-zero in $k$ (and $n \geq 2$). Then $R$ is Gorenstein, \cite[Theorem 1]{Wat}. Let $A$ be the localization of $R$ at its irrelevant maximal ideal.  Note the completion of $A$ is
 $k[[X_1, \ldots, X_n]]^G$ which is a domain.
 By \cite[Chapter 3, 3.4]{AR} we get that $G(\wh{A})$ is a finitely generated  abelian group. Also see \cite[Chapter 3, 5.4,5,5]{AR} for sufficient conditions which ensure $A$ is an isolated singularity.

 \s Let $k$ be an algebraically closed field of characteristic not equal to $2,3,5$. Let $S = k[X, Y, Z_1, \ldots, Z_{m}]_{(X, Y, Z_1, \ldots, Z_{m})}$. Let $A = S/(f)$ where $f$ is one of the equations defining a simple singularity (see \cite[8.8]{Y}). Then $\wh{A}$ is a simple singularity.  They are isolated singularities (and so are domains if $\dim A \geq 2$). The Grothendieck groups of all simple singularities are computed and they are finitely generated, see \cite[13.10]{Y}.

 \s Let $k$ denote a field of characteristic not equal to $2$. Let \\ $R = k[X_1, \ldots, X_n]_{(X_1, \ldots, X_n)}$. Let $f = \sum_{i=1}^{n} a_i X_i^2$ be a quadratic form with $a_i \neq 0$ for all $i$. Set $A = R/(f)$. Note $\wh{A} = \wh{R}/(f)$ is an isolated singularity (and so a domain if $n \geq 3$), see \cite[14.2]{Y}. Thus $A$ is also an isolated singularity. By \cite{BEH},
 $\wh{A}$ has finite representation type (i.e., there exists only finitely  many indecomposable MCM $\wh{A}$-modules up-to isomorphism); also see \cite[14.10]{Y}. It follows that $G(\wh{A})$ is a finitely generated abelian group, see \cite[13.2]{Y}.
\section{Appendix}
In this section we prove the following result
\begin{theorem}
\label{h-N-ft}
Let $(A,\m)$ be a \CM \ local ring of dimension $d \geq 1$. Assume $A$ is of essentially of finite type over a field $K$. Also assume $\wh{A}$, the completion of $A$ is a domain.
Then the henselization $A^h$ of $A$ is NOT of essentially finite type over $K$.
\end{theorem}
We believe Theorem \ref{h-N-ft} is already known. However we are unable to find a reference. As this result is crucial for us we give a proof.

\begin{lemma}\label{red-domain}
Let $(A,\m)$ be a \CM \ local ring of dimension $d \geq 1$. Assume $A$ is of essentially of finite type over a field $K$. Also assume that $A$ is a domain. Then $A = T_P$ for some affine domain $T$ over $K$ and $P$ a prime in $T$.
\end{lemma}
\begin{proof}
By assumption $A = W^{-1}R$ where $R = K[X_1, \ldots, X_n]/I$ and $W$ is a multiplicatively closed ideal in $R$. The maximal ideal $\m$ of $A$ is of the form $W^{-1}P$ where $P$ is a prime in $R$. It follows that $A = R_P$. Let $P_1$ be a minimal prime of $I$ such that $\height(P/P_1) = \dim A$. We have a a surjective map
$R \rt T = R/P_1$ which induces a surjection $ \epsilon \colon A\rt T_P$. As $\dim A = \dim T_P$ and $A$ is a domain it follows that $\epsilon$ is an isomorphism. The result follows.
\end{proof}
We now give
\begin{proof}[Proof of Theorem \ref{h-N-ft}]
By Lemma \ref{red-domain} we may assume that $A = T_P$ for some affine domain $T$ and a prime $P$ in $T$. Let $S = K[X_1, \ldots,X_n]$ be a Noether normalization of $T$. We can assume that $P\cap T = (X_1, \ldots, X_d)$ (see \cite[Chapter 2, Theorem 3.1]{K}). Let $B = S_{(X_1, \ldots, X_d)}$.
Then we note that there is a local map $B \rt A$ with zero dimensional fiber. This induces a local map $B^h \rt A^h$ (again with zero dimensional fiber). We note that $B^h$ is regular and $A^h$ is \CM. So by \cite[23.1]{M} $A^h$ is flat over $B^h$.  As $A^h$ is a subring of $\wh{A}$ it follows that $A^h$ is a domain. Let $L( E)$ be a quotient field of $A^h(B^h))$ respectively. Note we may consider $B^h$ to be a subring of $A^h$, see \cite[7.5]{M}. Then $L$ contains $E$. Also $E$ contains $k(X_1,\ldots, X_n)$ the quotient field of $B$.

Let the characteristic of $K$ be either zero or $p > 0$.\\
\emph{Claim:} For every prime $q \neq p$ the field $E$ contains $\sqrt[q]{1 + X_1}$.\\
Let $f_q(t) = t^q - (1+X_1)$. We note that $X_1$ is in the maximal ideal $\n$ of $B^h$. Modulo $\n$ we get that
$$ \ov{f_q(t)} = t^q - 1 = (t-1)(t^{q-1} + \cdots + t+ 1). $$
The polynomials $t-1$, $t^{q-1} + \cdots + t+ 1$ are co-prime in $(B^h/\n)[t]$.
As $B^h$ is Henselian, we get that there exists $F(t), G(t)$ monic such that $f_q(t) = F(t)G(t)$ with
$$\ov{F(t)} = t - 1 \quad \text{and} \quad \ov{G(t)} = t^{q-1} + \cdots + t+ 1. $$
Say $F(t) = t - u$. Then  $u \in B^h$ and $u^q = 1+X_1$. The claim follows.

Suppose if possible $A^h$ is of finite type over $K$. Then by \ref{red-domain} it follows that the quotient field $L$ of $A^h$ is a finite extension of $K(Y_1, \ldots, Y_m)$.
Say $Y_1, \ldots, Y_c$ is algebraically independent over $k(X_1,\ldots, X_n)$ and $Y_{c+1}, \ldots, Y_m$ are algebraic over $V = k(X_1,\ldots, X_n, Y_1, \ldots, Y_c)$. In particular $L$ is a finite extension of $V$. This is a contradiction as $f_q(t) = t^q - (1+X_1)$ is irreducible over $V$ for all $q \neq p$ (use Eisenstein's criterion and Gauss-Lemma).
\end{proof}

\end{document}